\documentclass[final,1p,times]{elsarticle}
\usepackage{amsmath}
\allowdisplaybreaks[4]
\usepackage{amsthm}
\usepackage{amscd}
\usepackage{amsfonts}
\usepackage{amssymb}
\usepackage{graphicx}
\usepackage[pagebackref, colorlinks, citecolor=green, linkcolor=red, urlcolor=blue]{hyperref}
\newtheorem{theorem}{Theorem}[section]

\newtheorem{lemma}[theorem]{Lemma}

\usepackage{mathrsfs}
\usepackage{titletoc}
\numberwithin{equation}{section}

\newcommand{\D}{\Delta}
\newcommand{\ra}{\rightarrow}

\newcommand{\f}{\frac}

\renewcommand{\l}{\lambda}

\newcommand{\be}{\begin{equation}}
\renewcommand{\ra}{\rightarrow}
\newcommand{\ee}{\end{equation}}
\newcommand{\bea}{\begin{eqnarray}}
\newcommand{\eea}{\end{eqnarray}}
\newcommand{\bna}{\begin{eqnarray*}}
\newcommand{\ena}{\end{eqnarray*}}
\renewcommand{\o}{\omega}

\newcommand{\iv}{\int_{V}}

\newcommand{\me}{\mathrm{e}}

\renewcommand{\le}{\left}
\newcommand{\ri}{\right}

\newcommand{\ve}{\vert}
\newcommand{\V}{\Vert}

\newcommand{\na}{\nabla}

\newcommand{\up}{\upsilon}

\journal{***}
\begin{document}
   
\begin{frontmatter}

\title{ Existence and asymptotic  behaviors of solutions to  Chern-Simons systems and equations  on finite graphs}

\author{Songbo Hou $^a$\corref{cor1}}
\ead{housb@cau.edu.cn}

\author{Xiaoqing Kong$^b$}
\ead{kxq@cau.edu.cn}

\address{$^{a,b}$Department of Applied Mathematics, College of Science, China Agricultural University,  Beijing, 100083, P.R. China}
\cortext[cor1]{Corresponding author: Songbo Hou}

\begin{abstract}

In this paper, we investigate a system of equations derived from the $\text{U}(1)\times \text{U}(1)$ Abelian Chern-Simons model:
\begin{eqnarray*}\left\{\begin{aligned}
		\Delta u &=\lambda\left(a(b-a)\mathrm{e}^u-b(b-a)\mathrm{e}^{\upsilon}+a^2\mathrm{e}^{2u}-ab\mathrm{e}^{2\upsilon}+b(b-a)\mathrm{e}^{u+\upsilon} \right)+4\pi\sum\limits_{j=1}^{k_1}m_j\delta_{p_j},\\
		\Delta \upsilon&=\lambda\left(-b(b-a)\mathrm{e}^u+a(b-a)\mathrm{e}^{\upsilon}-ab\mathrm{e}^{2u}+a^2\mathrm{e}^{2\upsilon}+b(b-a)\mathrm{e}^{u+\upsilon} \right)+4\pi\sum\limits_{j=1}^{k_2}n_j\delta_{q_j},
	\end{aligned}
	\right.
\end{eqnarray*}
on finite graphs. Here, $\lambda>0$, $b>a>0$, $m_j>0\, (j=1,2,\cdot\cdot\cdot,k_1)$, $n_j>0\,(j=1,2,\cdot\cdot\cdot,k_2)$, and $\delta_{p}$ denotes the Dirac delta mass at vertex $p$. We establish an iteration scheme and prove the existence of solutions. Additionally, we propose a novel method to derive the asymptotic behavior of solutions as $\lambda$ approaches infinity. This method is also applicable to the Chern-Simons system:
$$\left\{\begin{aligned}
	\Delta u &=\lambda\mathrm{e}^{\upsilon}(\mathrm{e}^{u}-1)+4\pi\sum\limits_{j=1}^{k_1}m_j\delta_{p_j},\\
	\Delta \upsilon&=\lambda\mathrm{e}^{u}(\mathrm{e}^{\upsilon}-1)+4\pi\sum\limits_{j=1}^{k_2}n_j\delta_{q_j},
\end{aligned}
\right.
$$
and the classical Chern-Simons equation:
\begin{eqnarray*} \Delta u=\lambda \mathrm{e}^u(\mathrm{e}^u-1)+4\pi\sum\limits_{j=1}^{N}\delta_{p_j}.
\end{eqnarray*}

\end{abstract}
\begin{keyword}  finite graph\sep Chern-Simons model\sep sub-solution, asymptotic  behavior
\MSC [2020] 35J47, 05C22
\end{keyword}
\end{frontmatter}
 %\tableofcontentsalso
\section{Introduction}

In this paper, we explore an elliptic system emerging from the $\text{U}(1)\times \text{U}(1)$ Abelian Chern-Simons model \cite{kim1993schrodinger,wilczek1992disassembling}, given by
\be\label{cso}\left\{\begin{aligned}
	\Delta u &=\l\le(a(b-a)\me^u-b(b-a)\me^{\up}+a^2\me^{2u}-ab\me^{2\up}+b(b-a)\me^{u+\up} \ri)+4\pi\sum\limits_{j=1}^{k_1}m_j\delta_{p_j},\\
	\Delta \up&=\l\le(-b(b-a)\me^u+a(b-a)\me^{\up}-ab\me^{2u}+a^2\me^{2\up}+b(b-a)\me^{u+\up} \ri)+4\pi\sum\limits_{j=1}^{k_2}n_j\delta_{q_j},
\end{aligned}
\right.
\ee
	on finite graphs. In this system, $\l>0$, $b>a>0$, $m_j>0,(j=1,2,\cdot\cdot\cdot,k_1)$, $n_j>0,(j=1,2,\cdot\cdot\cdot,k_2)$, and $\delta_{p}$ represents the Dirac delta mass at vertex $p$.
	
	Chern-Simons models provide explanations for various physical phenomena, such as aspects of particle physics, condensed matter physics, high-temperature superconductors, and the quantum Hall effect. Many studies have investigated system (\ref{cso}) within a domain $\Omega\subset \mathbb{R}^2$. Lin and Prajapat \cite{lin2009vortex} demonstrated the existence of a unique maximal solution and a mountain-pass type solution on a torus for the case $(a,b)=(0,1)$. For $a>b>0$ in (\ref{cso}), readers can refer to \cite{MR3198855,MR3168626}. Huang \cite{huang2021vortex} verified the existence of topological solutions, which represent the maximal solution to (\ref{cso}), in the case of $b>a>0$. Additional findings can be found in \cite{MR3390935,MR3592753,MR3543183} and the references contained therein.

Recently, there have been exciting developments in the study of elliptic equations with exponential nonlinearity on graphs. Grigor'yan et al. \cite{grigor2016kazdan} confirmed the existence of the Kazdan-Warner equation, denoted as

\be \label{kwe}\D u=c-h\me^u\ee
on graphs. This conclusion was derived through the application of variational calculus and a method involving upper and lower solutions. Ge \cite{MR3648273} re-investigated the Kazdan-Warner equation (\ref{kwe}) on graphs, specifically in the negative case, thereby augmenting the results in \cite{grigor2016kazdan}. Keller and Schwarz \cite{MR3776360} explored the Kazdan-Warner equation on canonically compactifiable graphs, deriving results similar to those found in finite cases. The $p$th Kazdan-Warner equation on graphs,

\be \label{kwp}\D_p u=c-h\me^u\ee
was studied in \cite{ge2020p,zhang2018p}. In the context of an infinite graph, the existence of a solution to the Kazdan-Warner equation

\be\label{kwi}
\D f=g-h\me^f
\ee
was proven using a heat flow method \cite{ge2018kazdan}. Using a variational method, Liu and Yang \cite{liu2020multiple} obtained multiple solutions to the Kazdan-Warner equation

\be \label{kwn}
\D u+\kappa+K_{\l }\me^u=0
\ee
on graphs in the negative case. Sun and Wang \cite{MR4416135} defined the Brouwer degree and offered new proofs for certain known existence results regarding Kazdan-Warner equations on finite graphs.

Turning to the Chern-Simons equations, Huang et al. \cite{huang2020existence} analyzed the equation

\be\label{csi1} \D u=\l \me^u(\me^u-1)+4\pi\sum\limits_{j=1}^{N}\delta_{p_j}\ee
on a finite graph, where $\l>0$ is a constant and $\delta_{p}$ represents the Dirac delta mass at vertex $p$. They successfully proved the existence of solutions for the non-critical case. The first author and Sun \cite{hou2022existence} examined the critical case for equation (\ref{csi1}) and also considered a more general Chern-Simons equation, a topic that was also addressed by L{\"u} and Zhong \cite{lu2021existence}.
Huang et al. \cite{huang2021mean} further probed into the existence of maximal solutions to a Chern-Simons system, establishing the existence of multiple solutions as well. For additional results, please refer to \cite{CHAO2023126787,chao2022existence, gao2022existence}.

Let $V$ represent the vertex set, and $E$ the edge set. A finite graph can be denoted as $G=(V,E)$. We suppose that $G$ is connected, signifying that any pair of vertices can be connected through a finite number of edges. The weight on the edge $xy\in E$ is defined by $\o_{xy}$, which is assumed to be symmetric, that is, $\o_{xy}=\o_{yx}$.
Denote by $\mu :V\ra \mathbb{R}^{+}$ a finite measure. We define the $\mu$-Laplace operator for any function $u:V\ra \mathbb{R}$ by
$$\D u(x)=\f{1}{\mu(x)}\sum\limits_{y\sim x}\omega_{xy}(u(y)-u(x)),$$
where $y\sim x$ implies that $xy\in E$. For any two functions $u$ and $\up$, the gradient form is defined by
\be\label{lpd}\Gamma(u,\up)=\f{1}{2\mu(x)}\sum\limits_{y\sim x}\omega_{xy}(u(y)-u(x))(\up(y)-\up(x)).\ee
When $u=\up$, we denote it as $\Gamma(u)=\Gamma(u,u)$. Define
$$\ve \na u\ve (x) =\sqrt {\Gamma (u)(x)}=\le( \f{1}{2\mu(x)}\sum\limits_{y\sim x}\omega_{xy}\le(u(y)-u(x)\ri)^2\ri)^{\f{1}{2}}.$$  
The integral over $V$ is defined by
$$\int_V ud\mu=\sum\limits_{x\in V}\mu(x)u(x),$$ 
applicable for any function $u:V\ra \mathbb{R}$.

In a manner similar to the Euclidean case, we introduce the Sobolev space 

$$W^{1,2}(V)=\le\{u\,\Big|\, u:V\rightarrow \mathbb{R},\,\iv\le(\ve \na u\ve^2 +u^2\ri)d\mu<+\infty\ri\}$$ characterized by the norm $$\V u\V_{W^{1,2}(V)}=\le(\iv\le(\ve \na u\ve^2 +u^2\ri)d\mu\ri)^{1/2}.$$ 
We denote by $V^{\mathbb{R}}$ the set of all real functions that map $V$ to $\mathbb{R}$, that is, $V^{\mathbb{R}}=\{u\,|\,u \,\text{is a real function}: V\ra\mathbb{R}\}$. Given the finiteness of $G$, we can deduce that $W^{1,2}(V)$ is synonymous with $V^{\mathbb{R}}$.

 We first consider the system (\ref{cso}) on a finite graph $G$.  Set \be\left\{\begin{aligned}
 	f_1(u,\up ) &=a(b-a)\me^u-b(b-a)\me^{\up}+a^2\me^{2u}-ab\me^{2\up}+b(b-a)\me^{u+\up} ,\\
 	f_2(u,\up)&=-b(b-a)\me^u+a(b-a)\me^{\up}-ab\me^{2u}+a^2\me^{2\up}+b(b-a)\me^{u+\up}.
 \end{aligned}
 \right.
 \ee
 Then we rewrite (\ref{cso}) as
 \be\label{csr}\left\{\begin{aligned}
 	\Delta u &=\l f_1(u,\up)+4\pi\sum\limits_{j=1}^{k_1}m_j\delta_{p_j},\\
 	\Delta \up&=\l f_2(u,\up)+4\pi\sum\limits_{j=1}^{k_2}n_j\delta_{q_j}.
 \end{aligned}
 \right.
 \ee
 Consider the system 
 \be\label{gfu}\left\{\begin{aligned}
 	\Delta u &=-\f{4\pi N_1}{\ve V\ve }+4\pi\sum\limits_{j=1}^{k_1}m_j\delta_{p_j},\\
 	\Delta \up&=-\f{4\pi N_2}{\ve V\ve }+4\pi\sum\limits_{j=1}^{k_2}n_j\delta_{q_j},
 \end{aligned}
 \right.
 \ee
 where $N_1=\sum_{j=1}^{k_1} m_j$, $N_2=\sum_{j=1}^{k_2} n_j$, and $\ve V\ve$ is the volume of $V$, i.e., $\ve V\ve=\int_V d\mu$. 

Indeed, when we examine orthogonality in terms of the standard scalar product, expressed as $\langle f,g\rangle=\int_v fgd\mu $, we derive the following relationship: $$\text {\rm ran}\Delta =\left( \text{\rm ker}\Delta\right)^\perp=\{\text{\rm const}\}^\perp.$$
Given that $\delta_{p}$ can be equated to a function that yields $\f{1}{\mu(p)}$ at point $p$ and zero elsewhere, and considering that the integral of the right-hand elements in system (\ref{gfu}) amounts to zero, we can deduce that system (\ref{gfu}) consistently provides a solution. Similar lines of reasoning can be applied to system  (\ref{ncs}) and equation (\ref{up1}).

Let $(u_0,\up_0)$ be a solution to (\ref{gfu}).
  If $(\hat{u}, \hat{\up})$ is a solution to (\ref{csr}), setting $(u,\up)=(\hat{u}, \hat{\up})-(u_0,\up_0)$, we see that 
 \be\label{che}\left\{\begin{aligned}
 	\Delta u &=\l f_1(u+u_0,\up+\up_0)+\f{4\pi N_1}{\ve V\ve },\\
 	\Delta \up&=\l f_2(u+u_0,\up+\up_0)+\f{4\pi N_2}{\ve V\ve }.
 \end{aligned}
 \right.
 \ee
 We get the following theorem. 
\begin{theorem}\label{tho}
Assume that $b>a>0$. 	There is 
	$\l_0>0$ 
	such that if $\l>\l_0$,
	\begin{enumerate}[(1)]
		\item  The system  (\ref{che}) yields a maximal solution pair $(u_{\l}, \up_{\l})$, interpreted in such a way that if $(u,\up)$ represents any other solution, it follows that $u\leq u_{\l}$ and $\up \leq \up_{\l}$ in  $V$.
		\item   As $\l\ra \infty$, $(u_{\l}, \up_{\l})\ra (-u_0,-\up_0)$.
		
		\item As $\l\ra \infty$, $\l f_1(u_{\l}+u_0,\up_{\l }+\up_0)\ra - 4\pi\sum\limits_{j=1}^{k_1}m_j\delta_{p_j},$ \;\;\;$\l f_2(u_{\l}+u_0,\up_{\l }+\up_0)\ra  -4\pi\sum\limits_{j=1}^{k_2}n_j\delta_{q_j}.$ 
	\end{enumerate}
\end{theorem}
Next, we consider the Chern-Simons system 

\be\label{ncs}\left\{\begin{aligned}
	\Delta u &=\lambda\mathrm{e}^{\up}(\mathrm{e}^{u}-1)+4\pi\sum\limits_{j=1}^{k_1}m_j\delta_{p_j},\\
	\Delta \up&=\lambda\mathrm{e}^{u}(\mathrm{e}^{\up}-1)+4\pi\sum\limits_{j=1}^{k_2}n_j\delta_{q_j},
\end{aligned}
\right.
\ee  
on a finite graph $G$. 
If $(\tilde{u},\tilde{\up})$ is a solution to (\ref{ncs}), letting $(u,\up)=(\tilde{u},\tilde{\up})-(u_0,\up_0)$, then we obtain 
 \be\label{css}\left\{\begin{aligned}
	\Delta u &=\l \me^{\up_0+\up}(e^{u_0+u}-1)+\f{4\pi N_1}{\ve V\ve },\\
	\Delta \up &=\l \me^{u_0+u}(e^{\up_0+\up}-1)+\f{4\pi N_2}{\ve V\ve }.
\end{aligned}
\right.
\ee
Huang et al.  \cite{huang2021mean} proved that there is a critical value $\l_*\geq \frac{4\pi\max\{N_1,N_2\}}{\ve V\ve}$ such that if $\l>\l_*$, the system  (\ref{css}) has a maximal solution $(u_{\l}, \up_{\l})$. 

We prove the following theorem. 
\begin{theorem}\label{thw}
We denote by  $(u_{\l}, \up_{\l})$  the maximal solution of the system  (\ref{css}). This is underscored by the condition that for every alternative solution $(u,\up)$, it holds true that $u\leq u_{\l}$ and $\up \leq \up_{\l}$ in  $V$.  Then we have 
		\begin{enumerate}
		
		\item    $(u_{\l}, \up_{\l})\ra (-u_0,-\up_0)$, as $\l\ra \infty$.
		
		\item  $\l \me^{\up_0+\up}(e^{u_0+u}-1)\ra - 4\pi\sum\limits_{j=1}^{k_1}m_j\delta_{p_j},$ \;\;\;$\l \me^{u_0+u}(e^{\up_0+\up}-1)\ra  -4\pi\sum\limits_{j=1}^{k_2}n_j\delta_{q_j},$  as $\l\ra \infty$.
	\end{enumerate}
\end{theorem}

Let $\bar{u}_0$ be a solution to 
\be \label{up1}\D u= -\f{4\pi N}{\ve V\ve}+4\pi \sum\limits_{j=1}^{N}\delta_{p_j}, \ee on a finite graph $G$.
 If $\hat{u}$ is a solution to (\ref{csi1}), writing $\hat{u}=\bar{u}_0+u$, then we get 
\be\label{mfe}
\D u=\l \me^{\bar{u}_0+u} (\me^{\bar{u}_0+u}-1)+\f{4\pi N}{\ve V\ve}.
\ee
The results in \cite{huang2020existence,hou2022existence} show that there exists a critical value $\l_c\geq\f{16\pi N}{\ve V\ve }$ such that $(\ref{mfe})$ admits a maximal solution if and only if $\l\geq \l_c$, where $\ve V\ve$ is the volume of $V$.

Analogous to Theorem \ref{tho}, we get the result for (\ref{mfe}). 
\begin{theorem}\label{tht}
Denote by $u_{\l }$ the maximal  solution to (\ref{mfe}) for $\l>\l_c$ in the sense that  if $u$ is any other solution, then 
$u\leq u_{\l}$ in $V$.  There holds that 
\begin{enumerate}[(1)]
	\item  $u_{\l}\ra -\bar{u}_0$ as $\l\ra \infty$.

	\item  $\l \me^{\bar{u}_0+u_{\l}}(\me^{\bar{u}_0+u_{\l}}-1)\ra -4\pi \sum\limits_{j=1}^{N}\delta_{p_j},$  as $\l\ra \infty$.
	
\end{enumerate}
\end{theorem}
We arrange the rest of the paper as follows. In Section 2, we establish the  iteration scheme for any sub-solution and get the monotonic sequence  by the similar methods to those in \cite{hou2022existence, huang2020existence,huang2021vortex}. That implies that if (\ref{che}) has a sub-solution, then it admits a solution. We prove the existence of the solution by constructing the sub-solution. Furthermore, we develop a novel method to study the asymptotic  behaviors and finish the proof of Theorem \ref{tho}. In Section 3, we prove Theorem \ref{thw}.  For Eq.(\ref{mfe}), the behaviors of the solutions  as $\l$ goes to the critical value have been investigated in \cite{hou2022existence}. So far,  no one has considered the behaviors of the solutions to (\ref{mfe}) as $\l$ goes to infinity. We study their  behaviors by the same method as those in proving Theorem \ref{tho} and prove   Theorem \ref{tht} in Section 4.

\section{Proof of Theorem \ref{tho}}

Firstly, we reference a key maximum principle, denoted as Lemma 4.1, as outlined in \cite{huang2020existence}. For comprehensive understanding, we also provide a proof of this principle herein.
\begin{lemma}\label{map}
	Let $G=(V,E)$ be a finite graph.
	If there is a positive constant $K$ such that  $\D  u(x) -Ku(x) \geq 0$ for all $x\in V$, there holds $u\leq 0$ in $V$.
\end{lemma}
\begin{proof} 

The proof unfolds via contradiction. Let's posit that $u>0$ at a certain point within $V$. Consequently, $u$ reaches its maximum at a point, denoted by $x_0$, where $u(x_0)>0$. This implies that
$$\D u(x_0) \geq Ku(x_0) >0.$$ 
However, according to the definition of the $\mu$-Laplace, it's established that $\D u(x_0)\leq 0$. This leads us to a contradiction. Therefore, it follows that $u\leq 0$ throughout $V$.
\end{proof}

Secondly, we  prove the following Lemma which will be used later. 

\begin{lemma}	\label{exs}
Suppose that $K$ is a positive constant. For any given $f\in W^{1,2}(V)$, the equation $\le(\D-K\ri)u=f$ has a solution $u$.
\end{lemma}
\begin{proof}

We adhere to the argument presented in the proof of Theorem 2.1 as delineated in \cite{ge2020p}. Our first step is to establish the functional
$$
	E(u)=\frac{1}{2} \int_V|\nabla u|^2 d \mu+\frac{1}{2} \int_V K u^2 d \mu+\int_V f u d \mu, \quad u \in W^{1,2}(V).
$$
The critical point of $E$ provides a solution for the equation $(\Delta - K)u=f$. We then contemplate the Euler-Lagrange equation associated with $E(u)$. Through our calculations, we ascertain
$$
	\left.\frac{d}{d t}\right|_{t=0} E(u+t \phi)=-\int_V\left(\Delta u-K u-f\right) \phi d \mu.
$$

This suggests that $\left.\frac{d}{d t}\right|_{t=0} E(u+t \phi)=0$ holds true if and only if $\Delta u -K u=f$. Next, we present that $E(u)$ approaches $+\infty$ as $\|u\|_{W^{1,2}(V)}$ also tends to $+\infty$. This signifies that $E(u)$ reaches its minimum within $W^{1,2}(V)$, which is a finite-dimensional linear space. Given the inequality
$$
	\left|\int_V f u d \mu\right| \leq\|f\|_2\|u\|_2\leq C_{f, G}\|u\|_{W^{1,2}(V)},
$$
we can then conclude
$$
\begin{aligned}
		 E(u) &\geq \frac{1}{2} \int_V|\nabla u|^2 d \mu+\frac{1}{2} \int_V K u^2 d \mu-C_{f, G}\|u\|_{W^{1,2}(V)} \\
		& \geq \f{\bar{K}}{2}\|u\|^2_{W^{1,2}(V)}-C_{f, G}\|u\|_{W^{1,2}(V)} \\
		& \rightarrow+\infty
		\end{aligned}
$$
as $\|u\|_{W^{1,2}(V)}\rightarrow+\infty$, where $\bar{K}=\min\{1,K\}$.  Supposing that $E(u)$ attains its minimum at some $u \in W^{1,2}(V)$, we then observe $\left.\frac{d}{d t}\right|_{t=0} E(u+t \phi)=0$, which implies $\Delta u-K u=f$.

\end{proof}

We call $(u_-,\up_-)$ a sub-solution to (\ref{che}) if it satisfies 
\be\label{chs}\left\{\begin{aligned}
	\Delta u_- &\geq \l f_1(u_-+u_0,\up_-+\up_0)+\f{4\pi N_1}{\ve V\ve },\\
	\Delta \up_-&\geq \l f_2(u_-+u_0,\up_-+\up_0)+\f{4\pi N_2}{\ve V\ve }.
\end{aligned}
\right.
\ee
\subsection{ Monotonic sequence} 
In order to get the existence of solutions to (\ref{che}), letting $(u_1,\up_1)=(-u_0,-\up_0)$, we first perform the following iteration scheme 
\be\label{ite}\left\{\begin{aligned}
	\le(\Delta-K\ri) u_{n+1} &= \l f_1(u_{n}+u_0,\up_{n}+\up_0)-Ku_{n}+\f{4\pi N_1}{\ve V\ve },\\
\le(\Delta-K\ri) \up_{n+1}&=\l f_2(u_{n}+u_0,\up_{n}+\up_0)-K\up_{n}+\f{4\pi N_2}{\ve V\ve },
\end{aligned}
\right.
\ee
where $K>2\l\le((a+b)(b-a)+2a^2\ri)$. 
\begin{lemma}\label{mol}

Let's assume that the pair $(u_-,\up_-)$ is a sub-solution to system (\ref{che}). Then we derive  that the sequence $(u_n,\up_n)$ is monotonic and has $(u_-,\up_-)$ as its lower bound, which can be expressed as:
\be\label{lmr}
\left\{\begin{aligned}
	u_1> u_2>\cdot\cdot\cdot> u_n>\cdot\cdot\cdot
	>u_{-},\\
	\up_1> \up_2>\cdot\cdot\cdot> \up_n>\cdot\cdot\cdot
	>\up_{-}.
\end{aligned}
\right.
\ee
Therefore, if system (\ref{che}) has a sub-solution, it will also admit a solution $(u_{\l}, \up_{\l})=\lim\limits_{n\ra \infty}(u_n,\up_n)$.

\end{lemma}
\begin{proof}
	
The proof follows from mathematical induction. Lemma \ref{exs} infers that for any given pair $(u_n,\up_n)$, there exists a subsequent pair $(u_{n+1},\up_{n+1})$ satisfying equation (\ref{ite}).

When $n=1$, given that $f_1(u_1+u_0,\up_1+\up_0)=f_1(0,0)=0$ and $f_2(u_1+u_0,\up_1+\up_0)=f_2(0,0)=0$, we arrive at the observation that:

\be\label{uvm}
\left\{\begin{aligned}
	\le(\Delta-K\ri)\le( u_2-u_1 \ri)&=4\pi\sum\limits_{j=1}^{k_1}m_j\delta_{p_j},\\
	\le(\Delta-K\ri)\le(\up_2-\up_1\ri)&=4\pi\sum\limits_{j=1}^{k_2}n_j\delta_{q_j}.
\end{aligned}
\right.
\ee
Applying Lemma \ref{map} to the first equation in (\ref{uvm}), we deduce that $u_2\leq u_1$. Noticing that $u_2-u_1$ only assumes a finite number of values, we establish the existence of some $x_0$ in $V$ such that $(u_2-u_1)(x_0)$ reaches its maximum. If $(u_2-u_1)(x_0)=0$, then (\ref{uvm}) implies:

\be
\D (u_2-u_1)(x_0)=4\pi\sum\limits_{j=1}^{k_1}m_j\delta_{p_j}\geq 0.
\ee
With respect to the definition of the $\mu-$Laplace, we arrive at $\D (u_2-u_1)(x_0)=0$, suggesting that $(u_2-u_1)(x)=0$ if $x\sim x_0$. Moreover, since $G$ is connected, for any $x\in V$, $(u_2-u_1)(x)=0$, which contradicts (\ref{uvm}). Therefore, we conclude that $u_1>u_2$. Likewise, $\up_1>\up_2$.

Assuming that $u_1> u_2>\cdot\cdot\cdot> u_n$ and $\up_1> \up_2>\cdot\cdot\cdot> \up_n$ for some $n$, it follows from the first equation in (\ref{ite}) that:
\bna\begin{aligned}
	(\D-K)(u_{n+1}-u_n)&= \l f_1(u_{n}+u_0,\up_{n}+\up_0)-\l f_1(u_{n-1}+u_0,\up_{n-1}+\up_0)-K(u_n-u_{n-1})\\
	&=\le(\l \f{\partial f_1}{\partial u}(\xi,\eta)-K\ri)(u_n-u_{n-1})+\l\f{\partial f_1}{\partial \up }(\xi,\eta)(\up_n-\up_{n-1}).\\
\end{aligned}
\ena
Here, $\xi$ is between $u_0+u_n$ and $u_0+u_{n-1}$, while $\eta$ is between $\up_0+\up_{n}$ and $\up_0+\up_{n-1}$. Taking into account:
\bna\begin{aligned}
	\l\f{\partial f_1}{\partial u}(\xi,\eta)&=\l\le(a(b-a)\me^{\xi}+2a^2\me^{2\xi}+b(b-a)\me^{\xi+\eta}\ri)\\
	&\leq \l\le(a(b-a)+2a^2+b(b-a)\ri)\\
	&=\l\le((a+b)(b-a))+2a^2\ri),
\end{aligned}
\ena
and
\bna\begin{aligned}
	\l\f{\partial f_1}{\partial \up }(\xi,\eta)&=\le(-b(b-a)\me^{\eta}-2ab\me^{2\eta}+b(b-a)\me^{\xi+\eta}\ri)\\
	&=\l\le(-2ab\me^{2\eta}-b(b-a)\me^{\eta}(1-\me^{\xi})\ri)<0,
\end{aligned}
\ena
we deduce:
\be\label{iti}
(\D-K)(u_{n+1}-u_n)>0.
\ee
Applying Lemma \ref{map} once again, we infer $u_{n+1}\leq u_n$. Following a similar process as in the case when $n=1$, we conclude $u_{n+1}<u_n$. Therefore, we have $u_1> u_2>\cdot\cdot\cdot> u_n>\cdot\cdot\cdot$. In a similar vein, we establish $\up_1> \up_2>\cdot\cdot\cdot> \up_n>\cdot\cdot\cdot$.

Next, we make a comparison between $(u_n, \up_n )$ and $(u_-, \up_-)$. Noting that $G$ only has a finite number of vertices, we can conclude that $x_u$ and $x_{\up}$ exist such that $ u_-+u_0$ and $\up_-+\up_0$ reach their maxima at $x_u$ and $x_{\up}$, respectively. Given equations (\ref{gfu}) and (\ref{chs}), along with the definition of the $\mu$-Laplace, we derive the following:

\begin{equation}\label{lsi}
	\begin{aligned}
		0\geq &\Delta  (u_--u_1)(x_u)\\
		\geq & f_1(u_{-}+u_0,\up_{-}+\up_0)(x_u)+4\pi\sum_{j=1}^{k_1}m_j\delta_{p_j}\\
		=&\lambda \left(a(b-a)\me ^{(u_--u_1)(x_u)}-b(b-a)\me^{(\up_--\up_1)(x_u)}+a^2\me^{2(u_--u_1)(x_u)}-ab\me^{2(\up_- -\up_1)(x_u)}\right.\\
		&+\left.b(b-a)\me^{(u_-+\up_--u_1-\up_1)(x_u)} \right) +4\pi\sum_{j=1}^{k_1}m_j\delta_{p_j}.
	\end{aligned}
\end{equation}
Assume $(u_--u_1)(x_u)>0$ and $(u_--u_1)(x_u)\geq (\up_--\up_1)(x_u)$. The right-hand side of equation (\ref{lsi}) can then be rewritten as

\begin{equation}\label{rwi}
	\begin{aligned}
		&\lambda \left((a(b-a)+ab\me^{(\up_--\up_1)(x_u)}+a^2\me^{(u_--u_1)(x_u)})(\me^{(u_--u_1)(x_u)}-\me^{(\up_- -\up_1)(x_u)})\right. \\
		&\left.+(b-a)^2\me^{(\up_- -\up_1)(x_u)}(\me^{(u_--u_1)(x_u)}-1)\right)+4\pi\sum_{j=1}^{k_1}m_j\delta_{p_j}>0,
	\end{aligned}
\end{equation}
which leads to a contradiction.

Assume $(u_--u_1)(x_u)>0$ and $(\up_--\up_1)(x_u)>(u_--u_1)(x_u)$. It is clear that $(\up_--\up_1)(x_{\up })\geq (u_--u_1)(x_u)>0$. Analogous to equation (\ref{lsi}), we obtain
\begin{equation}\label{lft}
	\begin{aligned}
		0\geq &\Delta  (\up_--\up_1)(x_{\up})\\
		\geq & \lambda f_2(u_{-}+u_0,\up_{-}+\up_0)(x_{\up})+4\pi\sum_{j=1}^{k_2}n_j\delta_{q_j}\\
		=&\lambda \left(-b(b-a)\me ^{(u_--u_1)(x_{\up })}+a(b-a)\me^{(\up_--\up_1)(x_{\up})}-ab\me^{2(u_--u_1)(x_{\up})}+a^2\me^{2(\up_- -\up_1)(x_{\up})}\right.\\
		&+\left.b(b-a)\me^{(u_-+\up_--u_1-\up_1)(x_{\up})} \right) +4\pi\sum_{j=1}^{k_2}n_j\delta_{q_j}. 
	\end{aligned}
\end{equation}
Again, the right-hand side of equation (\ref{lft}) is rewritten as
\begin{equation}
	\begin{aligned}
		&\lambda \left((a(b-a)+ab\me^{(u_--u_1)(x_{\up})}+a^2\me^{(\up_--\up_1)(x_u)})(\me^{(\up_--\up_1)(x_{\up})}-\me^{(u_- -u_1)(x_{\up})})\right. \\
		&\left.+(b-a)^2\me^{(u_- -u_1)(x_{\up})}(\me^{(\up_--\up_1)(x_{\up})}-1)\right)+4\pi\sum_{j=1}^{k_2}n_j\delta_{q_j}>0,
	\end{aligned}
\end{equation}
which results in a contradiction. Therefore, $(u_--u_1)(x_u)\leq 0$, and similarly, $(\up_--\up_1)(x_{\up})\leq 0$.

If $(u_--u_1)(x_u)=0$, using a similar technique to that in equations (\ref{lsi}) and (\ref{rwi}), we observe that

\begin{equation}
	0\geq \Delta  (u_--u_1)(x_u)\geq 4\pi\sum_{j=1}^{k_1}m_j\delta_{p_j}\geq 0.
\end{equation}
Thus, $(u_--u_1)(x)\equiv 0$ because $G$ is connected, leading to a contradiction. As a result, we have $u_-<u_1$, and similarly, $\up_-<\up_1$.

Now, let's assume that $u_n>u_-$ and $\up_n>\up_-$ for some $n$. According to equations (\ref{chs}) and (\ref{ite}), we obtain

\begin{equation}
	\begin{aligned}
		(\Delta-K)(u_{-}-u_{n+1})&\geq \lambda f_1(u_{-}+u_0,\up_{-}+\up_0)-\lambda f_1(u_{n}+u_0,\up_{n}+\up_0)-K(u_--u_{n})\\
		&=\left(\lambda \frac{\partial f_1}{\partial u}(\xi,\eta)-K\right)(u_--u_{n})+\lambda\frac{\partial f_1}{\partial \up }(\xi,\eta)(\up_--\up_{n}).\\
	\end{aligned}
\end{equation}
By applying a technique akin to the one used in proving equation (\ref{iti}), we deduce that

\begin{equation}
	(\Delta-K)(u_{-}-u_{n+1})>0.
\end{equation}
Using the same arguments as those in proving $u_{n+1}<u_n$, we demonstrate that $u_{n+1}>u_-$ and $\up_{n+1}>\up_-$. This completes the proof of Lemma \ref{mol}.

\end{proof}

\subsection{The sub-solution} 
Let $(u_{\l}, \up_{\l})$ be a solution to (\ref{che}) determined by Lemma \ref{mol}.  Noting that  a solution  to (\ref{che})  is also a sub-solution,  we  always  have $u_{\l}\geq u$ and $\up_{\l}\geq \up$ for any solution $(u,\up)$ to (\ref{che}). In this sense, we say that $(u_{\l}, \up_{\l})$ is a maximal solution. Next, we give  the existence of the sub-solution. 
\begin{lemma}
There exists $\l_0>0$ such that  the system (\ref{che}) has a sub-solution $(u_-, \up_-)$ for all $\l>\l_0$.
\end{lemma}
\begin{proof}

If $u=\up$, then $f_1(u,\up)$ and $f_2(u,\up)$ are simplified to $(a-b)^2(\me^{2u}-\me^{u})$. Let's set $c>0$ as a fixed constant, and define $(u_-,\up_- )=(-u_0-c,-\up_0-c)$. We observe that
\begin{equation}\label{suv}
	\left\{\begin{aligned}
		\Delta u_-&=\frac{4\pi N_1}{\ve V\ve}-4\pi\sum_{j=1}^{k_1}m_j\delta_{p_j},\\
		\Delta \up_- &=\frac{4\pi N_2}{\ve V\ve}-4\pi\sum_{j=1}^{k_2}n_j\delta_{q_j}.
	\end{aligned}
	\right.
\end{equation}
Given that
\[ f_1(u_-+u_0,\up_-+\up_0)=(a-b)^2(\me^{-2c}-\me^{-c}) \]
and
\[ f_2(u_-+u_0,\up_-+\up_0)=(a-b)^2(\me^{-2c}-\me^{-c}), \]
we conclude that
\begin{equation}\label{lso}
	\left\{\begin{aligned}
		\Delta u_-&\geq\lambda f_1(u_-+u_0,\up_-+\up_0 )+\frac{4\pi N_1}{\ve V\ve},\\
		\Delta \up_- &\geq \lambda f_2(u_-+u_0,\up_-+\up_0 )+\frac{4\pi N_2}{\ve V\ve},
	\end{aligned}
	\right.
\end{equation}
if $\lambda$ is large enough. Consequently, there exists $\lambda_0>0$ such that if $\lambda>\lambda_0$, $(u_-,\up_-)$ is a sub-solution to (\ref{che}).

\end{proof}
\subsection{Asymptotic  behavior}
\begin{lemma}\label{asb}
Let's consider $(u_{\l}, \up_{\l})$ as the maximal solution derived from Lemma \ref{mol}. With this assumption, the following conclusions hold true:

\begin{enumerate}[(1)]
	\item  As $\l\ra \infty$,  $(u_{\l}, \up_{\l })$ $\ra$ $(-u_0,-\up_0)$.
	
	\item  As $\l\ra \infty$, $\l f_1(u_{\l}+u_0,\up_{\l }+\up_0)$ $\ra$ $- 4\pi\sum\limits_{j=1}^{k_1}m_j\delta_{p_j}$ and  $\l f_2(u_{\l}+u_0,\up_{\l }+\up_0)$ $\ra$ $-4\pi\sum\limits_{j=1}^{k_2}n_j\delta_{q_j}$.
\end{enumerate}

\end{lemma}
\begin{proof}
Without loss of generality, let's assume that $N_2\geq N_1$. Let's define $\{\l_m\}$ as a sequence for which $\l_m\ra \infty$ as $m\ra \infty$. Also, let $\eta=\max\limits_{x\in V}\f{1}{\mu(x)}$, $\l_m>\f{16\pi N_2\eta}{(a-b)^2}$, and $(u_-,\up_- )=(-u_0-c,-\up_0-c)$ where $$c=-\ln \le(\f{1+\sqrt{1-\f{16\pi N_2\eta}{\l_m(a-b)^2}}}{2}\ri).$$

A straightforward verification indicates that

\be\left\{\begin{aligned}
	\Delta u_-&\geq\l_m f_1(u_-+u_0,\up_-+\up_0 )+\f{4\pi N_1}{\ve V\ve },\\
	\Delta \up_- &\geq \l_m f_2(u_-+u_0,\up_-+\up_0 )+\f{4\pi N_2}{\ve V\ve }.
\end{aligned}
\right.
\ee
Thus, $(u_-,\up_-)$ is a sub-solution to system (\ref{che}) with $\l=\l_m$, implying the existence of a solution $(u_{\l_m},\up_{\l_m})$ to (\ref{che}). By Lemma \ref{mol}, we have

$$-u_0+\ln \le(\f{1+\sqrt{1-\f{16\pi N_2\eta}{\l_m(a-b)^2}}}{2}\ri)\leq u_{\l_m}<-u_0$$
and
$$-\up_0+\ln \le(\f{1+\sqrt{1-\f{16\pi N_2\eta}{\l_m(a-b)^2}}}{2}\ri)\leq \up_{\l_m}<-\up_0. $$

From this, it is immediate that
$(u_{\l_m}, \up_{\l_m })\ra (-u_0,-\up_0)$ as $\l_m\ra \infty$.
Thus, we conclude $(u_{\l}, \up_{\l })\ra (-u_0,-\up_0)$ as $\l\ra \infty$, given that $\{\l_m\}$ may be any sequence for which $\l_m\ra \infty$.

Given any $\phi\in V^{\mathbb{R}}$, it can be deduced from (\ref{che}) that

\be\begin{aligned}
	\iv \l_m f_1(u_0+u_{\l_m}, \up_0+\up_{\l_m}) \phi d\mu&=\iv u_{\l_m} \D \phi d\mu-\f{4\pi N_1}{\ve V \ve }\iv \phi d\mu\\
	&\ra -\iv u_0 \D \phi d \mu-\f{4\pi N_1}{\ve V \ve }\iv \phi d\mu\\
	&=-\iv \D u_0 \phi  d\mu-\f{4\pi N_1}{\ve V \ve }\iv \phi d\mu\\
	&=-4\pi\sum\limits_{j=1}^{k_1}m_j\phi(p_j),
\end{aligned}
\ee
as $\l_m$ approaches infinity. Therefore,

$$\l_m f_1(u_0+u_{\l_m}, \up_0+\up_{\l_m}) \ra -4\pi\sum\limits_{j=1}^{k_1}m_j\delta_{p_j},$$ as $\l_m$ approaches infinity.
Similarly,

$$\l_m f_2(u_0+u_{\l_m}, \up_0+\up_{\l_m}) \ra -4\pi\sum\limits_{j=1}^{k_2}n_j\delta_{q_j},$$ as $\l_m$ approaches infinity.
Moreover,

$$\l f_1(u_{\l}+u_0,\up_{\l }+\up_0)\ra - 4\pi\sum\limits_{j=1}^{k_1}m_j\delta_{p_j}$$ and $$\l f_2(u_{\l}+u_0,\up_{\l }+\up_0)\ra  -4\pi\sum\limits_{j=1}^{k_2}n_j\delta_{q_j}, $$ as $\l\ra \infty$.

\end{proof}

\section {Proof of Theorem \ref{thw}}

  Assume that $N_2\geq N_1$. Let us consider a sequence $\{\l_m\}$ such that it approaches infinity, i.e., $\l_m\ra \infty$ as $m\ra \infty$. Furthermore, we assume that $\l_m>16\pi N_2\eta$ with $\eta=\max\limits_{x\in V}\f{1}{\mu(x)}$. Now, we define $(u_-,\up_- )=(-u_0-c,-\up_0-c)$ where
  
  $$
  c=-\ln \le(\f{1+\sqrt{1-\f{16\pi N_2\eta}{\l_m}}}{2}\ri).$$
  From this, we deduce

  \be\label{csi}\left\{\begin{aligned}
  	\Delta u_- &\geq \l_m \me^{\up_0+\up_-}(e^{u_0+u_-}-1)+\f{4\pi N_1}{\ve V\ve },\\
  	\Delta \up_- &\geq \l_m \me^{u_0+u_-}(e^{\up_0+\up_-}-1)+\f{4\pi N_2}{\ve V\ve },
  \end{aligned}
  \right.
  \ee
  which indicates that $(u_-,\up_- )$ is a sub-solution to (\ref{css}) at $\l=\l_m$. Referring to Lemma 3.2 in \cite{huang2021mean}, we can infer that (\ref{css}) admits a solution $(u_{\l_m }, \up_{\l_m})$ which is the maximal solution. Therefore, we find
  
 $$
  -u_0+\ln \le(\f{1+\sqrt{1-\f{16\pi N_2\eta}{\l_m}}}{2}\ri)\leq u_{\l_m}<-u_0$$
  and
  
 $$
  -\up_0+\ln \le(\f{1+\sqrt{1-\f{16\pi N_2\eta}{\l_m}}}{2}\ri)\leq \up_{\l_m}<-\up_0.$$
  As a result, as $\l_m\ra \infty$, $(u_{\l_m}, \up_{\l_m })$ converges to $(-u_0,-\up_0)$. Given that $\{\l_m\}$ can be any sequence with $\l_m\ra \infty$, we further conclude that $(u_{\l}, \up_{\l })$ converges to $(-u_0,-\up_0)$ as $\l\ra \infty$.
  
  For any $\phi\in V^{\mathbb{R}}$, using (\ref{css}), we can deduce:

  \be\begin{aligned}
  	\iv \l_m \me^{\up_0+\up_{\l_m}}(e^{u_0+u_{\l_m}}-1) \phi d\mu&=\iv u_{\l_m} \D \phi d\mu-\f{4\pi N_1}{\ve V \ve }\iv \phi d\mu\\
  	&\ra -\iv u_0 \D \phi d \mu-\f{4\pi N_1}{\ve V \ve }\iv \phi d\mu\\
  	&=-\iv \D u_0 \phi  d\mu-\f{4\pi N_1}{\ve V \ve }\iv \phi d\mu\\
  	&=-4\pi\sum\limits_{j=1}^{k_1}m_j\phi(p_j),
  \end{aligned}
  \ee
  as $\l_m\ra \infty$. Hence, we have
  
 $$
  \l_m \me^{\up_0+\up_{\l_m}}(e^{u_0+u_{\l_m}}-1)\ra -4\pi\sum\limits_{j=1}^{k_1}m_j\delta_{p_j}.$$
  Additionally, we observe that
  
$$
  \l_m \me^{u_0+u_{\l_m}}(e^{\up_0+\up_{\l_m}}-1)\ra -4\pi\sum\limits_{j=1}^{k_2}n_j\delta_{q_j},$$
  as $\l_m\ra \infty$. Hence, it implies
  
 $$
  \l \me^{\up_0+\up_{\l}}(e^{u_0+u_{\l}}-1)\ra -4\pi\sum\limits_{j=1}^{k_1}m_j\delta_{p_j}$$
  and
  
$$
  \l \me^{u_0+u_{\l}}(e^{\up_0+\up_{\l}}-1)\ra -4\pi\sum\limits_{j=1}^{k_2}n_j\delta_{q_j},$$
  as $\l\ra \infty$. This concludes the proof of Theorem \ref{thw}.

\section{Proof of Theorem \ref{tht}}
The methodology articulated in Lemma \ref{asb} can be effectively applied to the equation represented as

\be\label{CSH} \D u=\l \me^u(\me^u-1)+4\pi\sum\limits_{j=1}^{N}\delta_{p_j},\ee
which was previously explored in \cite{huang2020existence,hou2022existence}.

By invoking Lemma 4.2 from \cite{huang2020existence}, we ascertain that if Eq.(\ref{mfe}) possesses a sub-solution, it thereby admits a solution. Let's denote $u_{\l}$ as the maximal solution established by Lemma 4.2 in \cite{huang2020existence} with $u_-=-\bar{u}_0-c$.

We then define $\{\l_m\}$ as a sequence wherein $\l_m\ra \infty$ as $m\ra \infty$. Let $\l_m\geq 16\pi N\eta$ and $u_-=-\bar{u}_0-c$ where

$$c=-\ln\le( \f{1+\sqrt{1-\f{16\pi N\eta}{\l_m}}}{2}\ri) $$ and  $\eta=\max\limits_{x\in V}\f{1}{\mu(x)}$.
This leads us to

\be
\D u_-\geq \l_m \me^{\bar{u}_0+u_-}(\me^{\bar{u}_0+u_-}-1)+\f{4\pi N}{\ve V\ve }.
\ee
Applying Lemma 4.2 from \cite{huang2020existence}, we infer that

$$-\bar{u}_0+\ln\le( \f{1+\sqrt{1-\f{16\pi N\eta}{\l_m}}}{2}\ri)\leq u_{\l_m}<-\bar{u}_0.$$
Consequently, we deduce
$u_{\l_m}\ra -\bar{u}_0$ as $m\ra \infty$. Furthermore, $u_{\l}\ra -\bar{u}_0$ as $\l\ra \infty$.

For any $\phi\in V^{\mathbb{R}}$, it is discernible from (\ref{mfe}) that

\be\begin{aligned}
	\iv \l_m \me^{\bar{u}_0+u_{\l_m}}(\me^{\bar{u}_0+u_{\l_m}}-1) \phi d\mu&=\iv u_{\l_m} \D \phi d\mu-\f{4\pi N}{\ve V \ve }\iv \phi d\mu\\
	&\ra -\iv \bar{u}_0 \D \phi d \mu-\f{4\pi N}{\ve V \ve }\iv \phi d\mu\\
	&=-\iv \D \bar{u}_0 \phi  d\mu-\f{4\pi N}{\ve V \ve }\iv \phi d\mu\\
	&=-4\pi\sum\limits_{j=1}^{N}\phi(p_j),
\end{aligned}
\ee 
as $\l_m$ approaches infinity.
This leads us to the conclusion that

$$\l \me^{\bar{u}_0+u_{\l}}(\me^{\bar{u}_0+u_{\l}}-1)\ra -4\pi \sum\limits_{j=1}^{N}\delta_{p_j},$$  
as $\l\ra \infty$.
With this, we successfully conclude the proof of Theorem \ref{tht}.

\vskip 20 pt
\noindent{\bf Acknowledgement}

This work is  partially  supported by the National Natural Science Foundation of China (Grant No. 11721101), and by National Key Research and Development Project SQ2020YFA070080.

\vskip 12 pt
\noindent{\bf Competing Interests} 

The authors have no competing interests to declare that are relevant to the content of this article.

\vskip 12 pt
\noindent{\bf  Data Availability }
 
Data sharing not applicable to this article as no datasets were generated or analysed during the current study.

\vskip 20 pt

\bibliographystyle{plain}
\bibliography{D:/Reference/CHS.bib}

\end{document}